\documentclass{amsart}
\usepackage{amssymb,amstext,amsmath,amscd,amsthm,amsfonts,enumerate,graphicx,latexsym,stmaryrd,multicol}

\usepackage[usenames]{color}
\usepackage[all]{xy}

\usepackage[top=34truemm,bottom=15truemm,left=25truemm,right=25truemm]{geometry}

\newtheorem{thm}{Theorem}[section]
\newtheorem{lem}[thm]{Lemma}

\newtheorem{cor}[thm]{Corollary}
\theoremstyle{definition}
\newtheorem{dfn}[thm]{Definition}
\newtheorem{ques}[thm]{Question}

\newtheorem{ex}[thm]{Example}
\newtheorem{nota}[thm]{Notation}
\newtheorem{rmk}[thm]{Remark}
\theoremstyle{remark}

\newtheorem*{ac}{Acknowledgments}

\numberwithin{equation}{thm}

\def\Cok{\operatorname{Coker}}

\def\depth{\operatorname{depth}}
\def\Ext{\operatorname{Ext}}

\def\Gdim{\operatorname{G-dim}}
\def\ge{\geqslant}

\def\grade{\operatorname{grade}}

\def\Hom{\operatorname{Hom}}

\def\image{\operatorname{Im}}
\def\J{\mathrm{J}}

\def\le{\leqslant}

\def\m{\mathfrak{m}}

\def\p{\mathfrak{p}}

\def\r{\operatorname{r}}

\def\syz{\Omega}

\def\Tor{\operatorname{Tor}}
\def\tr{\operatorname{Tr}}

\begin{document}
\allowdisplaybreaks
\title[Ext modules related to syzygies of the residue field]{Ext modules related to syzygies of the residue field}
\author{Yuya Otake}
\address{Graduate School of Mathematics, Nagoya University, Furocho, Chikusaku, Nagoya 464-8602, Japan}
\email{m21012v@math.nagoya-u.ac.jp}

\begin{abstract}
Let $R$ be a commutative noetherian ring.
In this paper, we find out close relationships between the module $M$ being embedded in a module of projective dimension at most $n$ and the $(n+1)$-torsionfreeness of the $n$th syzygy of $M$.
As an application, when $R$ is local with residue field $k$, we compute the dimensions as $k$-vector spaces of Ext modules related to syzygies of $k$.
\end{abstract}
\maketitle
\section{Introduction}
Throughout this paper, let $R$ be a commutative noetherian ring.
We assume that all modules are finitely generated ones.
Auslander and Bridger \cite{AB} established the theory of the stable category of finitely generated modules over noetherian rings.
The notion of $n$-torsionfree modules, which is a natural generalization of the notion of torsionfree modules over integral domains, plays an important role in stable module theory.
Among other things, Auslander and Bridger studied when $n$-syzygy modules become $n$-torsionfree and proved the following approximation theorem.

\begin{thm}[Auslander--Bridger]\label{abapp}
Let $M$ be an $R$-module and $n\ge0$ an integer. The following are equivalent.
\begin{enumerate}[\rm(1)]
   \item
   The $n$th syzygy $\syz^n M$ of $M$ is $n$-torsionfree.
   \item
   There exists an exact sequence $0\to B\to J\to M\to 0$ of $R$-modules such that $B$ has projective dimension at most $n-1$ and $\Ext^i_R(J,R)=0$ for all integers $1\le i\le n$.
\end{enumerate}
\end{thm}

Simon \cite{Sim} called the exact sequence appearing in Theorem \ref{abapp}(2) the {\it $n$-approximation} of $M$ and its dual version the {\it $n$-hull} of $M$.
As a consequence of Theorem \ref{abapp}, we see that any module over a Gorenstein local ring has a {\it Cohen--Macaulay approximation} in the sense of Auslander and Buchweitz \cite{ABu}.
It is also proved in \cite{ABu} that over a Gorenstein local ring any module admits a {\it finite projective hull}, which is a dual notion of a Cohen--Macaulay approximation.
Namely, for any module $M$ over a Gorenstein local ring $R$, there exists an exact sequence $0\to M\to W\to C\to 0$ of $R$-modules such that $W$ has finite projective dimension and $C$ is maximal Cohen--Macaulay.
In particular, every module over a Gorenstein local ring can be embedded in a module of finite projective dimension. 
Conversely, Foxby \cite{Fox} proved that if $R$ is a Cohen--Macaulay local ring and every $R$-module can be embedded in an $R$-module of finite projective dimension, then $R$ is Gorenstein.
Takahashi, Yassemi and Yoshino \cite{TYY} succeeded in removing from Foxby's theorem the assumption of Cohen--Macaulayness of the ring $R$.

In the present paper, for a fixed integer $n$, we begin with considering embedding a given module in a module of projective dimension at most $n$.
Namely, our first aim is to explore the following question.

\begin{ques}\label{when}
Let $M$ be an $R$-module and $n$ a nonnegative integer.
When can $M$ be embedded in an $R$-module of projective dimension at most $n$?
\end{ques}

Our answer to Question \ref{when} is the following theorem, which says that the question is closely related to the $(n+1)$-torsionfreeness of $n$th syzygies.

\begin{thm}[Theorem \ref{FPD}]\label{introthm1}
Let $M$ be an $R$-module and $n\ge0$ an integer.
If $\syz^n M$ is $(n+1)$-torsionfree, then $M$ can be embedded in an $R$-module of projective dimension at most $n$.
The converse is also true if $M_{\p}$ has finite Gorenstein dimension over $R_{\p}$ for all prime ideals $\p$ of $R$ with $\depth R_{\p}<n$.
 \end{thm}
As an application of the above theorem, we recover the result of Takahashi, Yassemi and Yoshino in Corollary \ref{TYYcor}.

Next, let us consider the case where $R$ is local with residue field $k$, and has depth $t$.
Recently, Dey and Takahashi \cite{DT} studied the torsionfreeness of syzygies of $k$.
They especially proved in \cite[Theorems 4.1(2) and 4.5(1)]{DT} that $\syz^t k$ is $(t+1)$-torsionfree, and it is a $(t+2)$nd syzygy if and only if the local ring $R$ has type one.
Motivated by their results, as another application of Theorem \ref{introthm1}, we consider the length of the $R$-module $\Ext^i_R(\tr\syz^j k, M)$ for an $R$-module $M$ and positive integers $i, j$, where $\tr(-)$ is the (Auslander) transpose.
We prove that the module $\Ext^i_R(\tr\syz^t k, M)$ is a $k$-vector space for all integers $i>0$ if $M$ has depth at least $t$, and compute its dimension as a $k$-vector space.
By letting $M=R$ in our computations, we obtain the following corollary.


\begin{cor}[Corollary \ref{DTtype}]\label{introDTtype}
Let $R$ be a local ring with residue field $k$ and have depth $t$.
Then the following hold.
\begin{enumerate}[\rm(1)]
\item
   The syzygy $\syz^t k$ is $(t+2)$-torsionfree if and only if the local ring $R$ has type one.
   \item
  One has $\Ext^i_R(\tr\syz^t k,R)=0$ for some $i\ge t+3$ if and only if the ring $R$ is Gorenstein.
\end{enumerate}
\end{cor}

If $\syz^t k$ is $(t+2)$-torsionfree, then in particular it is a $(t+2)$nd syzygy, so $R$ has type one by the result of Dey and Takahashi mentioned above.
Corollary \ref{introDTtype}(1) asserts that the converse also holds.
In other words, the equivalent conditions given by Dey and Takahashi are also equivalent to the $(t+2)$-torsionfreeness of $\syz^t k$.
Furthermore, it follows from Corollary \ref{introDTtype}(2) that if $\syz^t k$ is $(t+3)$-torsionfree, then $R$ is Gorenstein.

The organization of this paper is as follows.
In Section 2, we consider when a given module can be embedded in a module of finite projective dimension.
We prove in this section the above Theorem \ref{introthm1}, and state some corollaries and examples.
In Section 3, when $R$ is local, we prove that some Ext modules related to the residue field of $R$ are vector spaces over the residue field, and compute their dimensions as an application of a result obtained in the previous section, which produces the above Corollary \ref{introDTtype}.
\section{Modules embedded in modules of finite projective dimension}
In this section, we consider that when a given $R$-module can be embedded in an $R$-module of finite projective dimension.
The main result of this section gives an answer to this question and is used to compute the lengths of Ext modules in the next section.

First, we introduce some basic notations and notions.
\begin{nota}
\begin{enumerate}[\rm(1)]
  \item
Let $M$ and $N$ be $R$-modules.
By $M\approx N$ we mean that there are projective modules $P$ and $Q$ such that $M\oplus P\cong N\oplus Q$.
   \item
We denote by $(-)^\ast$ the $R$-dual functor $\Hom_R(-,R)$. 
\end{enumerate}
\end{nota}

The following notions were introduced in \cite{AB} and play an important role in this paper.

\begin{dfn}
\begin{enumerate}[\rm(1)]
   \item
Let $M$ be an $R$-module and $P_1\xrightarrow{\partial} P_0\to M\to0$ a projective presentation of $M$.
The {\em (first) syzygy $\syz M$} and {\em (Auslander) transpose $\tr M$} of $M$ are defined as $\image\partial$ and $\Cok \partial^\ast$, respectively.
The modules $\syz M$ and $\tr M$ are uniquely determined by $M$ up to projective summands.
For any integer $n>0$ the {\em $n$th syzygy $\syz^n M$ of $M$} is defined inductively by $\syz^n M=\syz(\syz^{n-1}M)$.
   \item \cite[Definition 2.15]{AB}
Let $n\ge0$ be an integer.
An $R$-module $M$ is said to be {\em $n$-torsionfree} if $\Ext^i_R(\tr M,R)=0$ for all integers $1\le i\le n$.
An $R$-module $M$ is called {\em torsionless} if it is $1$-torsionfree and called {\em reflexive} if it is $2$-torsionfree.
   \item \cite[Definition 2.35]{AB}
Let $n>0$ be an integer.
An $R$-module $M$ is said to be {\em $n$-spherical} if $\Ext^i_R(M,R)=0$ for all integers $1\le i\le n-1$ and $M$ has projective dimension at most $n$.
   \item \cite[Definition 3.7]{AB}
An $R$-module $M$ is said to be {\em Gorenstein projective} if $\Ext^i_R(M,R)=0$ and $\Ext^i_R(\tr M,R)=0$ for all integers $i>0$.
   The {\em Gorenstein dimension} of an $R$-module $M$ is defined to be the infimum of integers $n$ such that there exists an exact sequence $0\to G_n\to G_{n-1}\to\cdots\to G_1\to G_0\to M\to0$ of $R$-modules with $G_i$ Gorenstein projective, and denoted by $\Gdim_R M$.
\end{enumerate}
\end{dfn}

The following theorem is the main result of this section.
This theorem describes the relationship between Question \ref{when} and various conditions including the $(n+1)$-torsionfreeness of $n$th syzygies.


\begin{thm}\label{FPD}
Let $M$ be an $R$-module and $n$ a nonnegative integer.
Consider the following conditions.
\begin{enumerate}[\rm(1)]
    \item
    The module $\syz^n M$ is $(n+1)$-torsionfree.
    \item
    There exists an exact sequence $0\to M\to W\to C\to 0$ of $R$-modules such that $W$ has projective dimension at most $n$ and $C\approx \tr\syz^{n+1}\tr\syz^n M$.
    \item
    There exists an exact sequence $0\to M\to W\to C\to 0$ of $R$-modules such that $W$ has projective dimension at most $n$ and $\Ext^i(C,R)=0$ for all $1\le i\le n+1$.
    \item
    The module $M$ can be embedded in an $R$-module of projective dimension at most $n$.
    \item
    The module $\syz^n M$ is an $(n+1)$th syzygy.
\end{enumerate}
Then the implications $(1)\Longleftrightarrow(2)\Longleftrightarrow(3)\Longrightarrow(4)\Longrightarrow(5)$ hold.
If $\Gdim_{R_{\p}}M_{\p}<\infty$ for all prime ideals $\p$ of $R$ with $\depth R_{\p}<n$, then all the five conditions are equivalent.
\end{thm}

\begin{proof}
Assume that $\syz^n M$ is $(n+1)$-torsionfree.
In particular, since $\syz^n M$ is $n$-torsionfree, there exists an exact sequence $0\to B\to \J_n(M)\to M\to 0$ such that $B$ has projective dimension at most $n-1$ and $\J_n(M)\approx\tr\syz^n\tr\syz^n M$ by \cite[Proposition 2.21]{AB}.
As $\J_n(M)$ is torsionless by the assumption, we have an exact sequence $0\to\J_n(M)\to P\to\tr\syz\tr\J_n(M)\to0$ with $P$ projective.
Put $C=\tr\syz\tr\J_n(M)$.
We construct the pushout diagram:
$$
\xymatrix@R-1pc@C-1pc{
&&0\ar[d]&0\ar[d]&\\
0\ar[r]&B\ar@{=}[d]\ar[r]&\J_n(M)\ar[d]\ar[r]&M\ar[r]\ar[d]&0\\
0\ar[r]&B\ar[r]&P\ar[d]\ar[r]&W\ar[d]\ar[r]&0\\
&&C\ar[d]\ar@{=}[r]&C\ar[d]\\
&&0&0.
}
$$
Note that $W$ has projective dimension at most $n$ and
$$
C\cong\tr\syz\tr\J_n(M)\cong\tr\syz\tr\tr\syz^n\tr\syz^n M\approx\tr\syz^{n+1}\tr\syz^n M.
$$
Thus the implication $(1)\Rightarrow(2)$ follows.
The implication $(2)\Rightarrow(3)$ follows from \cite[Proposition 1.1.1]{I}.
Suppose that the condition (3) holds.
Then there exists an exact sequence $0\to M\to W\to C\to 0$ such that $W$ has projective dimension at most $n$ and $\Ext^i(C,R)=0$ for all integers $1\le i\le n+1$.
Applying the horseshoe lemma to this repeatedly, we get an exact sequence $0\to M^{\prime}\to W^{\prime}\to C^{\prime}\to0$ such that $M^{\prime}\approx\syz^n M$, $W^{\prime}\approx\syz^n W$ and $C^{\prime}\approx\syz^n C$.
As $W^{\prime}$ is projective, we have $\syz^n M\approx\syz^{n+1}C$.
By \cite[Proposition 1.1.1]{I}, $\syz^n M$ is $(n+1)$-torsionfree and thus the implication $(3)\Rightarrow(1)$ holds.
The implication $(3)\Rightarrow(4)$ is clear.
The implication $(4)\Rightarrow(5)$ follows from taking $\syz^n$ as in the proof of the implication $(3)\Rightarrow(1)$.
Under the assumption that $\Gdim_{R_{\p}}M_{\p}<\infty$ for all prime ideals $\p$ of $R$ with $\depth R_{\p}<n$, if $\syz^n M$ is an $(n+1)$th syzygy, then by \cite[Theorem 43]{Masek} $\syz^n M$ is $(n+1)$-torsionfree.
\end{proof}

Simon \cite{Sim} called the exact sequence appearing in Theorem \ref{FPD}(3) the $(n+1)$-hull of $M$.
The existence of $n$-hulls is also discussed in \cite[Theorem 2.18]{Sim}.
The minimality of $n$-approximations and $n$-hulls is studied in \cite[Section 3]{Sim}.

We can deduce a result of Takahashi, Yassemi and Yoshino \cite[Corollary 2.4]{TYY} directly from Theorem \ref{FPD}.

\begin{cor}[Takahashi--Yassemi--Yoshino]\label{TYYcor}
Suppose that $R$ is local and with depth $t$.
Let $k$ be the residue field of $R$.
Then the following are equivalent.
\begin{enumerate}[\rm(1)]
\item
The ring $R$ is Gorenstein.
\item
Any $R$-module can be embedded in an $R$-module of finite projective dimension.
\item
The module $\tr\syz^t k$ can be embedded in an $R$-module of finite projective dimension.
\end{enumerate}
\end{cor}

\begin{proof}
Assume that $R$ is Gorenstein.
Then for any $R$-module $M$ the $t$th syzygy $\syz^t M$ is maximal Cohen--Macaulay.
By \cite[Theorem 3.3.10]{BH}, $\syz^t M$ is $(t+1)$-torsionfree and the implication $(1)\Rightarrow(2)$ follows from Theorem \ref{FPD}.
The implication $(2)\Rightarrow(3)$ is clear.
Suppose that $\tr\syz^t k$ is a submodule of an $R$-module of finite projective dimension.
It follows from Theorem \ref{FPD} that $\syz^t\tr\syz^t k$ is $(t+1)$-torsionfree.
In particular, $\Ext^{1}(\syz^t\tr\syz^t\tr\syz^t k,R)=\Ext^{t+1}(\tr\syz^t\tr\syz^t k,R)=0$.
Since $\Ext^1(\syz^t k,R)$ is a direct summand of $\Ext^{1}(\syz^t\tr\syz^t\tr\syz^t k,R)$, we have $\Ext^{t+1}(k,R)=\Ext^1(\syz^t k,R)=0$ and the implication $(3)\Rightarrow(1)$ follows from \cite[II. Theorem 2]{Rob}.
\end{proof}

The following corollary is necessary to prove Theorem \ref{t+1}, which is one of the main theorems in this paper.
For a local ring $(R, \m, k)$ we denote by $\r(R)$ the {\em type} of $R$, that is, $\r(R)$ is the dimension of the vector space $\Ext_R^{\depth R}(k,R)$ over the residue field $k$ of $R$.

\begin{cor}\label{kapp}
Suppose that $R$ is local and with depth $t$.
Let $k$ be the residue field of $R$.
Then there exists an exact sequence $0\to k\to W\to C\to0$ such that $W$ has projective dimension $t$ and $C\approx\tr\syz^{t+1}\tr\syz^t k$.
Moreover, if $t>0$, then $W\approx \tr\syz^{t-1}(k^{\oplus\r(R)})$.
\end{cor}

To give a proof of the above corollary, we need the following lemma, which follows immediately from \cite[Theorem 2.7(1)]{O2}.

\begin{lem}\cite[Theorem 2.7]{O2}\label{sph}
Let $W$ be an $R$-module and $s$ a positive integer.
If $W$ is $s$-spherical, then $W\approx\tr\syz^{s-1}\Ext^s_R(W,R)$.
\end{lem}

\begin{proof}[Proof of Corollary \ref{kapp}]
By \cite[Theorem 4.1(2)]{DT}, $\syz^t k$ is $(t+1)$-torsionfree.
It follows from Theorem \ref{FPD} that there exists an exact sequence $0\to k\to W\to C\to 0$ such that $W$ has projective dimension at most $t$ and $C\approx\tr\syz^{t+1}\tr\syz^t k$.
We assume that $t$ is positive.
Then since $\Ext^i(k,R)=0=\Ext^i(C,R)$ for all $1\le i\le t-1$, so does $W$. 
Hence $W$ is $t$-spherical and $\Ext^t(W,R)\cong\Ext^t(k,R)\cong k^{\oplus\r(R)}$.
By Lemma \ref{sph}, we obtain that $W\approx\tr\syz^{t-1}\Ext^t(W, R)\cong\tr\syz^{t-1}(k^{\oplus \r(R)})$.
\end{proof}

Grades of $\Ext$ modules are one of the main subjects of the theory of Auslander and Bridger; see \cite[Chapters 2 and 4]{AB}.
Recall that the {\em grade} of an $R$-module $M$ is defined to be the infimum of integers $i$ such that $\Ext^i_R(M,R)\ne0$, and denoted by $\grade_R M$.
We state the relationship between Theorem \ref{FPD} and the grade condition given by Auslander and Bridger.

\begin{cor}\label{ext}
Let $n\ge0$ be an integer and $M$ an $R$-module.
If $\syz^n M$ is $(n+1)$-torsionfree, then $\grade_R\Ext^i_R(M,R)\ge i$ for all integers $1\le i\le n$.
\end{cor}

\begin{proof}
By Theorem \ref{FPD}, there exists an exact sequence $0\to M\to W\to C\to 0$ of $R$-modules such that $W$ has projective dimension at most $n$ and $\Ext^i(C,R)=0$ for all $1\le i\le n+1$.
Hence we have $\Ext^i(M,R)\cong\Ext^i(W,R)$ for all $1\le i\le n$.
Since $W$ has finite projective dimension, $\grade\Ext^i(M,R)=\grade\Ext^i(W,R)\ge i$ for all $1\le i\le n$ by \cite[Corollary 4.17]{AB}.
\end{proof}



We close this section by stating an example for Theorem \ref{FPD}.
The following example asserts that the implication $(4)\Longrightarrow(1)$ in Theorem \ref{FPD} does not hold in general.

\begin{ex}
Let $R=k[\![x,y,z]\!]/(x^2, xy, y^2 z)$.
Then $R$ is a $1$-dimensional Cohen--Macaulay local ring with $y+z$ a parameter.
Here we construct an $R$-module $M$ such that $M$ is not torsionless (i.e. cannot be embedded in any free module), can be embedded in a module with projective dimension one, and $\syz M$ is not reflexive.
In this case, the module $M$ satisfies the condition $(4)$ in Theorem \ref{FPD} but does not satisfy the condition $(1)$.

We denote by $f:R\to (y)\oplus(z)$ the $R$-homomorphism given by $f(1)=\binom{y}{z}$.
Note that the homomorphism $\binom{y}{z}:R\to R^2$ is injective by the Hilbert--Burch theorem \cite[Theorem 1.4.17]{BH}, and hence $f$ is also injective.
Put $M=\Cok f$ and $W=\Cok\binom{y}{z}$.
We obtain the commutative diagram with exact rows and columns

$$
\xymatrix@R-1pc@C-1pc{
&&0\ar[d]&0\ar[d]&\\
0\ar[r]&R\ar@{=}[d]\ar[r]^-{f}&(y)\oplus(z)\ar[d]^-{\iota}\ar[r]&M\ar[r]\ar[d]&0\\
0\ar[r]&R\ar[r]^-{\binom{y}{z}}&R^2\ar[d]\ar[r]&W\ar[d]\ar[r]&0\\
&&{R/(y)}\oplus{R/(z)}\ar[d]\ar@{=}[r]&{R/(y)}\oplus{R/(z)}\ar[d]\\
&&0&0,
}
$$
where $\iota:(y)\oplus(z)\to R^2$ is the natural inclusion.
Applying $\Hom_R(k,-)$ to the first row of the commutative diagram above, we have the exact sequence
$$
0=\Hom_R(k,R^2)\to\Hom_R(k,W)\to\Ext^1_R(k,R)\xrightarrow{0}\Ext^1_R(k,R^2).
$$
Since $R$ has type $2$, $\Hom_R(k,W)$ is isomorphic to $k^{\oplus2}$.
Also, from the second column of the diagram above, we have the exact sequence
$$
0\to\Hom_R(k,M)\to\Hom_R(k,W)\to\Hom_R(k, {R/(y)})\oplus\Hom_R(k,{R/(z)}).
$$
Note that $\depth_R{R/(y)}>0$, $\depth_R{R/(z)}=0$ and $R/(z)$ has type 1.
We conclude that $\Hom_R(k,M)\ne0$, that is, $\depth_R M=0$.
In particular, $M$ is not torsionless.

Next, we show that $\syz M$ is not reflexive.
Since $\syz M\approx\syz^2(R/(y)) \oplus\syz^2(R/(z))$, it suffices to prove that $\syz^2(R/(y))$ is not reflexive.
By \cite[Proposition 2.26]{AB}, it is equivalent to proving the equality $\grade_R\Ext^2_R({R/(y)},R)=0$.
Put $\p=(x, y)$.
The ideal $\p$ is a minimal prime of $R$, and $R_{\p}\cong k[\![x,y,z]\!]_{(x, y)}/(x^2, xy, y^2)$ is an artinian non-Gorenstein local ring.
So we have
$$
\Ext^1_{R_\p}((x)R_{\p}, R_{\p})\oplus\Ext^1_{R_\p}((y)R_{\p}, R_{\p})\cong\Ext^1_{R_\p}((x, y)R_{\p}, R_{\p})\cong\Ext^2_{R_\p}(R_{\p}/{\p R_{\p}}, R_{\p})\ne0.
$$
We conclude that $\Ext^1_{R_\p}((y)R_{\p}, R_{\p})$ is nonzero by symmetry.
Now $0\ne\Ext^1_{R_\p}((y)R_{\p}, R_{\p})\cong\Ext^2_{R_\p}(R_{\p}/{(y)R_{\p}}, R_{\p})\cong\Ext^2_R(R/(y),R)_{\p}$, and therefore $\grade_R\Ext^2_R({R/(y)},R)=0$.

\end{ex}


\section{Dimensions of Ext modules related to the residue field}
In this section, when $R$ is local, we compute the lengths of Ext modules related to the residue field of $R$, using results obtained in the previous section.
For this we need the following lemma.
\begin{lem}\label{exttor}
Let $n$ be a positive integer, and let $M$ and $N$ be $R$-modules.
If $\Ext^i_R(M,R)=0$ for all integers $1\le i\le n-1$, then $\Ext^j_R(M,N)\cong\Tor^R_{n-j}(\tr\syz^{n-1}M,N)$ and $\Tor^R_j(M,N)\cong\Ext^{n-j}_R(\tr\syz^{n-1}M,N)$ for all integers $1\le j\le n-1$.
Moreover, if $M$ is $n$-spherical, then $\Ext^j_R(M,N)\cong\Tor^R_{n-j}(\Ext^n_R(M,R),N)$ and $\Tor^R_j(M,N)\cong\Ext^{n-j}_R(\Ext^n_R(M,R),N)$ for all integers $1\le j\le n$.
\end{lem}

\begin{proof}
Let $P_n\to P_{n-1}\to\cdots\to P_1\to P_0\to M\to 0$ be a projective resolution of $M$.
Since $\Ext^i(M,R)=0$ for all $1\le i\le n-1$, the $R$-dual $P_0^\ast \to P_1^\ast \to \cdots\to P_{n-1}^\ast \to P_n^\ast \to\tr\syz^{n-1}M\to0$ is exact.
We have the complex isomorphism
$$
\xymatrix@R-1pc@C-1pc{
P_0^\ast \otimes N\ar[r]\ar[d]^{\cong} & P_1^\ast \otimes N\ar[r]\ar[d]^{\cong} & \cdots\ar[r] & P_{n-1}^\ast \otimes N\ar[r]\ar[d]^{\cong} & P_n^\ast \otimes N\ar[d]^{\cong} \\
\Hom(P_0,N)\ar[r] & \Hom(P_1,N)\ar[r] & \cdots\ar[r] & \Hom(P_{n-1},N)\ar[r] & \Hom(P_n,N).
}
$$
Therefore, $\Ext^j(M,N)\cong\Tor_{n-j}(\tr\syz^{n-1}M,N)$ for all $1\le j\le n-1$.
If $M$ is $n$-spherical, then $\tr\syz^{n-1}M\approx\Ext^n(M,R)$; see \cite[Theorem 2.7]{O2}.
We only need to prove the case where $j=n$.
In this case, by the complex isomorphism above, we have $\Ext^n(M,N)\cong\Cok(\Hom(P_{n-1},N)\to\Hom(P_n,N))\cong\Cok(P_{n-1}^\ast \otimes N\to P_n^\ast \otimes N)\cong\Ext^n(M,R)\otimes N$. 
The dual statement is seen similarly.
\end{proof}

Let $R$ be a local ring with maximal ideal $\m$ and $M$ an $R$-module.
For all integers $i$, the $i$-th {\em Betti number} $\beta_i(M)$ (resp. {\em Bass number} $\mu_i(M)$) of $M$ is defined as $\dim_{R/\m}\Tor^R_i(R/\m, M)$ (resp. $\dim_{R/\m}\Ext^i_R(R/\m, M)$).
We note that if $i<0$, then $\beta_i(M)=\mu_i(M)=0$ for all modules $M$.

The following assertion is a corollary of the above lemma.

\begin{lem}\label{let}
Let $(R,\m,k)$ be local with depth $t>0$.
Let $M$ be an $R$-module and $1\le j\le t$ an integer.
Then, for all integers $i>0$, the modules $\Ext^i_R(\tr\syz^{j-1}k,M)$ and $\Tor^R_i(\tr\syz^{j-1}k,M)$ are $k$-vector spaces of dimensions $\beta_{j-i}(M)$ and $\mu_{j-i}(M)$, respectively.
\end{lem}

\begin{proof}
Let $i>0$ be an integer.
Since $\tr\syz^{j-1}k$ is $j$-spherical, if $j<i$, then $\Ext^i(\tr\syz^{j-1}k,M)$ and $\Tor_i(\tr\syz^{j-1}k,M)$ are the zero modules.
We may assume that $j\ge i$.
Note that $\Ext^j(\tr\syz^{j-1}k,R)\cong k$ by Lemma \ref{sph}.
Using Lemma \ref{exttor}, we obtain the isomorphisms
$$
\Ext^i(\tr\syz^{j-1}k,M)\cong\Tor_{j-i}(\Ext^j(\tr\syz^{j-1}k,R),M)\cong\Tor_{j-i}(k,M)\cong k^{\oplus \beta_{j-i}(M)}
$$
and the isomorphisms
$$
\Tor_i(\tr\syz^{j-1}k,M)\cong\Ext^{j-i}(\Ext^j(\tr\syz^{j-1}k,R),M)\cong\Ext^{j-i}(k,M)\cong k^{\oplus \mu_{j-i}(M)}.
$$\end{proof}




For a short exact sequence, if its $R$-dual is also exact, then the following holds for the long exact sequences of Tor and Ext.
This lemma is necessary to compute the dimensions of Ext modules in the proof of the main result of this section.
 \begin{lem}\label{connec}
 Let $0\to A\xrightarrow{f}B\xrightarrow{g}C\to0$ be an exact sequence of $R$-modules.
 Suppose that the $R$-dual $0\to C^\ast\xrightarrow{g^\ast}B^\ast\xrightarrow{f^\ast}A^\ast\to0$ is also exact.
 Then for any $R$-module $X$ there exists a long exact sequence
 $$
 \cdots\to\Tor_1^R(A,X)\to\Tor_1^R(B,X)\to\Tor_1^R(C,X)\to\Ext^1_R(\tr A,X)\to\Ext^1_R(\tr B,X)\to\Ext^1_R(\tr C,X)\to\cdots.
 $$
 \end{lem}

 \begin{proof}
By the horseshoe lemma, we have the commutative diagram
$$
\xymatrix@R-1pc@C-1pc{
0\ar[r]&P_1\ar[d]\ar[r]&P_1\oplus Q_1\ar[d]\ar[r]&Q_1\ar[d]\ar[r]&0\\
0\ar[r]&P_0\ar[d]\ar[r]&P_0\oplus Q_0\ar[d]\ar[r]&Q_0\ar[d]\ar[r]&0\\
0\ar[r]&A\ar[d]\ar[r]&B\ar[d]\ar[r]&C\ar[d]\ar[r]&0\\
&0&0&0&
}
$$
with exact rows and columns, where $P_i$ and $Q_i$ are projective for $i=0, 1$.
Since $0\to C^\ast\xrightarrow{g^\ast}B^\ast\xrightarrow{f^\ast}A^\ast\to0$ is exact, the $R$-dual diagram
$$
\xymatrix@R-1pc@C-1pc{
&0\ar[d]&0\ar[d]&0\ar[d]&\\
0\ar[r]&C^\ast\ar[d]\ar[r]&B^\ast\ar[d]\ar[r]&A^\ast\ar[d]\ar[r]&0\\
0\ar[r]&Q_0^\ast\ar[d]\ar[r]&(P_0\oplus Q_0)^\ast\ar[d]\ar[r]&P_0^\ast\ar[d]\ar[r]&0\\
0\ar[r]&Q_1^\ast\ar[d]\ar[r]&(P_1\oplus Q_1)^\ast\ar[d]\ar[r]&P_1^\ast\ar[d]\ar[r]&0\\
0\ar[r]&T_C\ar[d]\ar[r]^{t_g}&T_B\ar[d]\ar[r]&T_A\ar[d]\ar[r]&0\\
&0&0&0&
}
$$
has exact rows and columns, where $T_A\approx\tr A$, $T_B\approx\tr B$ and $T_C\approx\tr C$.
Taking $\Hom(-,X)$, we obtain a commutative diagram with exact rows and columns
$$
\xymatrix@R-1pc@C-1pc{
&0\ar[d]&0\ar[d]&0\ar[d]&\\
0\ar[r]&\Hom(T_A, X)\ar[d]\ar[r]&\Hom(T_B,X)\ar[d]\ar[r]&\Hom(T_C,X)\ar[d]&\\
0\ar[r]&\Hom(P_1^\ast,X)\ar[d]\ar[r]&\Hom((P_1\oplus Q_1)^\ast,X)\ar[d]\ar[r]&\Hom(Q_1^\ast,X)\ar[d]\ar[r]&0\\
0\ar[r]&\Hom(P_0^\ast,X)\ar[d]\ar[r]&\Hom((P_0\oplus Q_0)^\ast,X)\ar[d]\ar[r]&\Hom(Q_0\ast,X)\ar[d]\ar[r]&0\\
&A\otimes X\ar[d]\ar[r]&B\otimes X\ar[d]\ar[r]&C\otimes X\ar[d]\ar[r]&0\\
&0&0&0.&
}
$$
By the snake lemma, we have an exact sequence
$$
0\to\Hom(T_A,X)\to\Hom(T_B,X)\xrightarrow{\Hom(t_g,X)}\Hom(T_C,X)\to A\otimes X\xrightarrow{f\otimes X} B\otimes X\to C\otimes X\to 0.
$$
Also, there are long exact sequences
$$
0\to\Hom(T_A,X)\to\Hom(T_B,X)\xrightarrow{\Hom(t_g,X)}\Hom(T_C,X)\to\Ext^1(\tr A,X)\to\Ext^1(\tr B,X)\to\cdots
$$
and
$$
\cdots\to\Tor_1(A,X)\to\Tor_1(B,X)\to\Tor_1(C,X)\to A\otimes X\xrightarrow{f\otimes X} B\otimes X\to C\otimes X\to 0.
$$
By connecting these sequences, we have the desired long exact sequences.
\end{proof}

Using the results obtained in this section and Corollary \ref{kapp}, we can give a proof of the following theorem, which is the main result of this section.
The following theorem can be regarded as a generalization of some results about the $n$-torsionfreeness of syzygies of the residue field of a local ring $R$; see Corollaries \ref{M=R} and \ref{DTtype}.
 
\begin{thm}\label{t+1}
Let $(R,\m,k)$ be a local ring of depth $t$ and $M$ an $R$-module of depth at least $t$.
Then, the module $\Ext^i_R(\tr\syz^t k,M)$ is a $k$-vector space for all integers $i>0$ and the following hold.
\begin{enumerate}[\rm(1)]
\item
One has $\dim_k \Ext^i_R(\tr\syz^t k,M)=\beta_{t-i+1}(M)$ for all integers $1\le i\le t$.
\item
One has $\dim_k \Ext^{t+1}_R(\tr\syz^t k,M)-\dim_k \Ext^{t+2}_R(\tr\syz^t k,M)=\beta_0(M)-r(R)\mu_t(M)$.
\item
One has $\dim_k\Ext^{i}_R(\tr\syz^t k,M)=\r(R)\mu_{i-2}(M)$ for all integers $i\ge t+3$.
\end{enumerate}
\end{thm}

\begin{proof}
Assume that $t=0$.
There exists an exact sequence $0\to\Ext^1(\tr k,M)\to k\otimes M\to\Hom(k^\ast, M)\to\Ext^2(\tr k,M)\to0$ by \cite[Proposition 2.6(a)]{AB}.
Since $k\otimes M\cong k^{\oplus\beta_0(M)}$ and $\Hom(k^\ast ,M)\cong \Hom(k^{\oplus\r(R)},M)\cong k^{\oplus\r(R)\mu_0(M)}$, the assertion (2) follows.
Also, as $\syz^2 \tr k\approx k^\ast \cong k^{\r(R)}$, we have $\Ext^i(\tr k,M)\cong\Ext^{i-2}(\syz^2 \tr k,M)\cong k^{\oplus\r(R)\mu_{i-2}(M)}$ for all $i>2$.
In this case, we have the desired result.

Next, we handle the case where $t$ is positive.
By Corollary \ref{kapp}, there exists an exact sequence $0\to k\to W\to C\to 0$ such that $W\approx\tr\syz^{t-1}(k^{\oplus \r(R)})$ and $C\approx\tr\syz^{t+1}\tr\syz^t k$.
There exists an exact sequence
$$
\Tor_{i+1}(W,M)\to\Tor_{i+1}(C,M)\to\Tor_{i}(k,M)\to\Tor_{i}(W,M)
$$
for any integer $i$.
As $W$ is $t$-spherical, $\Tor_i(W,M)=0$ for all $i>t$.
Moreover, by Lemma \ref{exttor}, we have the isomorphisms
$$
\Tor_i(W,M)\cong\Ext^{t-i}(\Ext^t(W,R),M)\cong\Ext^{t-i}(\Ext^t(\tr\syz^{t-1}k,R),M)^{\oplus\r(R)}\cong\Ext^{t-i}(k,M)^{\oplus\r(R)}
$$
for any integer $1\le i\le t$.
By these isomorphisms and the assumption on $M$, we obtain that $\Tor_i(W,M)=0$ for all integers $1\le i\le t$.
Consequently, $\Tor_i(W,M)=0$ for all $i>0$, that is, $\Tor_{i+1}(C,M)\cong\Tor_{i}(k,M)\cong k^{\oplus\beta_{i}(M)}$ for all $i>0$.
Note that $\syz^{t}k$ is $(t+1)$-torsionfree and therefore $\Ext^i(C,R)=0$ for all integers $1\le i\le t+1$ by \cite[Proposition 1.1.1]{I}.
We have
$$
\tr\syz^{t+1}C\approx\tr\syz^{t+1}\tr\syz^{t+1}\tr\syz^t k\cong\tr\syz^t k.
$$
Using Lemma \ref{exttor}, we obtain that
$$
\Tor_i(C,M)\cong\Ext^{t+2-i}(\tr\syz^{t+1}C,M)\cong\Ext^{t+2-i}(\tr\syz^t k,M)
$$
for all integers $1\le i\le t+1$.
Hence $\Ext^i(\tr\syz^t k,M)\cong\Tor_{t+2-i}(C,M)\cong k^{\oplus\beta_{t+1-i}(M)}$ for all integers $1\le i\le t$ and the assertion (1) follows.
Also, the exact sequence $0\to k\to W\to C\to0$ satisfies the assumption of Lemma \ref{connec}.
There exists a long exact sequence
\begin{equation}\label{long}
\Tor_1(W,M)\to\Tor_1(C,M)\to\Ext^1(\tr k,M)\to\Ext^1(\tr W,M)\to\Ext^1(\tr C,M)\to\cdots.
\end{equation}
Note that $\Ext^i(\tr W,M)\cong\Ext^{t-1+i}(k,M)^{\oplus\r(R)}\cong k^{\oplus\r(R)\mu_{t-1+i}(M)}$ and $\Ext^i(\tr C,M)\cong\Ext^{t+1+i}(\tr\syz^t k,M)$ for all $i>0$.
Moreover, since $\tr k$ has projective dimension at most one, $\Ext^i(\tr k,M)=0$ for all $i>1$.
By the long exact sequence (\ref{long}), we have $\Ext^{t+1+i}(\tr\syz^t k,M)\cong\Ext^i(\tr C,M)\cong\Ext^i(\tr W,M)\cong k^{\oplus \r(R)\mu_{t-1+i}(M)}$ for all $i>1$.
The assertion (3) holds.
By Lemma \ref{let}, $\Ext^1(\tr k,M)\cong k^{\oplus\beta_0(M)}$.
Since $\Tor_1(W,M)=0$ and $\Tor_1(C,M)\cong\Ext^{t+1}(\tr\syz^t k,M)$, we have the exact sequence $0\to \Ext^{t+1}(\tr\syz^t k,M)\to k^{\beta_0(M)}\to k^{\oplus\r(R)\mu_t(M)}\to\Ext^{t+2}(\tr\syz^{t}k,M)\to0$ by (\ref{long}). 
The assertion (2) follows.
\end{proof}

\begin{rmk}\label{t+1rmk}
Let $M$ be an $R$-module and $I$ an ideal of $R$.
We denote by $(0\underset{M}{:}I)$ the submodule of $M$ consisting of elements $x$ of $M$ such that $Ix=0$.
Suppose that $R$ is local with residue field $k$, and has depth 0.
For the proof of Theorem \ref{t+1} in the case $t=0$, the map $\sigma_{k, M}:k\otimes_R M\to\Hom_R(k^\ast, M)$ is given by $\sigma_{k, M}(a\otimes x)(f)=f(a)x$ for $a\in k$, $x\in M$ and $f\in k^\ast$.
It is easily seen that the image of the map $\sigma_{k, M}$ is isomorphic to the module $M/{(0 \underset{M}{:} (0 \underset{R}{:} \m))}$, where $\m$ is the maximal ideal of $R$.
Therefore, by the proof of Theorem \ref{t+1}, the equalities $\dim_k\Ext^1_R(\tr k, M)=\beta_0(M)-\dim_k M/{(0 \underset{M}{:} (0 \underset{R}{:} \m))}$ and $\dim_k\Ext^2_R(\tr k, M)=\r(R)\mu_0(M)-\dim_k M/{(0 \underset{M}{:} (0 \underset{R}{:} \m))}$ hold. 
 \end{rmk}

Letting $M=R$ in Theorem \ref{t+1}, we obtain the following corollary.

\begin{cor}\label{M=R}
Let $(R,\m,k)$ be local and with depth $t$.
The module $\Ext^i_R(\tr\syz^t k,R)$ is a $k$-vector space for all integers $i>0$ and the following hold.
\begin{enumerate}[\rm(1)]
   \item
   One has $\Ext^i_R(\tr\syz^t k,R)=0$ for all integers $1\le i\le t+1$.
   \item
   One has $\dim_k \Ext^{t+2}_R(\tr\syz^t k,R)=\r(R)^2 -1$.
   \item
   One has $\dim_k\Ext^{i}_R(\tr\syz^t k,R)=\r(R)\mu_{i-2}(R)$ for all integers $i\ge t+3$.
\end{enumerate}
\end{cor}

\begin{proof}
The assertion (3) immediately follows from Theorem \ref{t+1}(3).
Note that $\syz^t k$ is $(t+1)$-torsionfree and therefore $\Ext^{t+1}(\tr\syz^t k,R)=0$.
The assertions (1) and (2) follow from Theorem \ref{t+1}(1)(2).
\end{proof}

The equivalence of (1-a) and (1-b) in the following corollary was proved by Dey and Takahashi \cite[Theorem 4.5(1)]{DT}.
Using the above corollary, we can prove that these conditions are also equivalent to the condition (1-c).
Moreover, we can state a criterion for the local ring $R$ to be Gorenstein.

\begin{cor}\label{DTtype}
Let $(R,\m,k)$ be local and with depth $t$.
The following hold.
\begin{enumerate}[\rm(1)]
   \item
   The following are equivalent.
   \begin{itemize}
   \item[(1-a)]   The local ring $R$ has type one.
   \item[(1-b)]   The module $\syz^t k$ is $(t+2)$-syzygy.
   \item[(1-c)]   The module $\syz^t k$ is $(t+2)$-torsionfree.
\end{itemize}
\item
One has $\Ext^i_R(\tr\syz^t k,R)=0$ for some integer $i\ge t+3$ if and only if $R$ is Gorenstein.
\end{enumerate}
\end{cor}

\begin{proof}
(1) The equivalence $\text{(1-a)}\Leftrightarrow\text{(1-b)}$ follows from  \cite[Theorem 4.5(1)]{DT}.
The equivalence $\text{(1-a)}\Leftrightarrow\text{(1-c)}$ follows from Corollary \ref{M=R}(1)(2).

(2) If $\Ext^i(\tr\syz^t k,R)=0$ for some integer $i\ge t+3$, then $\mu_{i-2}(R)=0$ by Corollary \ref{M=R}(3).
By \cite[II. Theorem 2]{Rob}, we conclude that $R$ is Gorenstein.
The converse is clear.
\end{proof}



\begin{ac}
The author is grateful to his supervisor Ryo Takahashi for his many helpful comments.
\end{ac}


\begin{thebibliography}{99}
\bibitem{AB}
{\sc M. Auslander; M. Bridger}, Stable module theory, Memoirs of the American Mathematical Society {\bf 94}, {\em American Mathematical Society, Providence, R.I.}, 1969.
\bibitem{ABu}
{\sc M. Auslander; R.-O. Buchweitz}, The homological theory of maximal Cohen--Macaulay approximations, Colloque en l'honneur de Pierre Samuel (Orsay, 1987), {\it M\'em. Soc. Math. France (N.S.)} {\bf 38} (1989), 5--37.
\bibitem{BH}
{\sc W. Bruns; J. Herzog}, Cohen--Macaulay rings, revised edition, Cambridge Studies in Advanced Mathematics {\bf 39}, {\it Cambridge University Press, Cambridge}, 1998.
\bibitem{DT}
{\sc S. Dey; R. Takahashi}, On the subcategories of $n$-torsionfree modules and related modules, {\em Collect. Math.} {\bf 74} (2023), no. 1, 113--132.
\bibitem{Fox}
{\sc H.-B.Foxby}, Embedding of modules over Gorenstein rings, {\em Proc. Amer. Math. Soc.} {\bf 36} (1972), 336--340.
\bibitem{I}
{\sc O. Iyama}, Higher-dimensional Auslander--Reiten theory on maximal orthogonal subcategories, {\em Adv. Math.} {\bf 210} (2007), no. 1, 22--50.
\bibitem{Masek} 
{\sc V. Ma\c{s}ek}, Gorenstein dimension and torsion of modules over commutative Noetherian rings, {\emph Comm. Algebra} {\bf 20} (2000), no. 12, 5783--5812.
\bibitem{O2}
{\sc Y. Otake}, Stable categories of spherical modules and torsionfree modules, {\em Proc. Amer. Math. Soc.} (to appear), {\tt arXiv:2204.04398}.
\bibitem{Rob}
{\sc P. Roberts}, Two applications of dualizing complexes over local rings, {\em Ann. Sci. \'Ecole Norm. Sup. (4)} {\bf 9} (1976), no. 1, 103--106.
\bibitem{Sim}
{\sc A. M. Simon}, About $q$-approximations and $q$-hulls over a Noetherian ring, some refinements of the Auslander-Bridger theory, {\em Comm. Algebra} {\bf 47} (2019), no. 11, 4496--4519.
\bibitem{TYY}
{\sc R. Takahashi; S. Yassemi; Y. Yoshino}, On the existence of embeddings into modules of finite homological dimensions, {\em Proc. Amer. Math. Soc.} {\bf 138} (2010), no. 7, 2265--2268.
\end{thebibliography}
\end{document}